\documentclass{article}
\usepackage[%
journal=XXX,    
lang=british,   
]{ems-journal}

\numberwithin{equation}{section}

\usepackage{amsmath}
\usepackage{graphicx}
\usepackage{epstopdf}

\usepackage{graphicx}%
\usepackage{multirow}%
\usepackage{amsmath}%
\usepackage{amsthm}%
\usepackage{mathrsfs}%
\usepackage[title]{appendix}%
\usepackage{xcolor}%
\usepackage{textcomp}%
\usepackage{manyfoot}%
\usepackage{booktabs}%
\usepackage{algorithm}%
\usepackage{algorithmicx}%
\usepackage{algpseudocode}%
\usepackage{listings}%

\newtheorem{thm}{Theorem}[section]
\newtheorem{cor}[thm]{Corollary}
\newtheorem{lem}[thm]{Lemma}

\theoremstyle{definition}
\newtheorem{defn}[thm]{Definition}
\newtheorem{rem}[thm]{Remark}

 \newcommand\R{\mathbb{R}}

 \newcommand\E{\mathbb{E}}

 \DeclareMathOperator*{\var}{Var}

\begin{document}

\title{On the uniform Besov regularity of local times of general processes}
\titlemark{Besov regularity of local times}


\emsauthor{1}{
	\givenname{Brahim}
	\surname{Boufoussi}
	\mrid{}
	\orcid{}}{B.~Boufoussi}
\emsauthor{2}{
	\givenname{Yassine}
	\surname{Nachit}
	\mrid{}
	\orcid{}}{Y.~Nachit}

\Emsaffil{1}{
	\department{Department of Mathematics}
	\organisation{Faculty of Sciences Semlalia, Cadi Ayyad University}
	\rorid{}
	\address{2390}
	\zip{40000}
	\city{Marrakesh}
	\country{Morocco}
	\affemail{boufoussi@uca.ac.ma}}
\Emsaffil{2}{
	\department{1}{Department of Mathematics}
	\organisation{1}{CNRS, Université de Lille}
	\rorid{1}{}
	\address{1}{Laboratoire Paul Painlevé UMR 8524}
	\zip{1}{F-59655}
	\city{1}{Villeneuve d’Ascq}
	\country{1}{France}
	\affemail{yassine.nachit@univ-lille.fr}}



\begin{abstract}
  The $\alpha$-local nondeterminism notion ($\alpha$-LND, for short) has been introduced by \cite{boufoussi2021local} to investigate the existence, joint continuity, and uniform  H\"{o}lder continuity, i.e., H\"{o}lder continuity in the time variable $t$ uniformly in the
  space variable $x$,  for local times $L(x, t)$  of general processes. In the present paper, we aim to use the $\alpha$-LND property  to improve this uniform H\"{o}lder continuity of local times to a uniform Besov regularity, i.e., Besov regularity in the time variable $t$ uniformly in the space variable $x$ and in $p$ (the index of the Besov space $\mathbf{B}^{\nu}_{p,\infty}(I;\R)$).
  The Besov regularity of local times, in the time variable $t$ for fixed space variable $x$, has never been treated in the literature even for Gaussian or stable processes.
  We also extend the classical Adler's theorem \cite[Theorem 8.7.1]{adler1981geometry} to the Besov spaces case.
  These results are then exploited to study the uniform (in $p$) Besov irregularity of the sample paths of the underlying processes.
  As applications, we get sharp uniform Besov irregularity results for a class of Gaussian processes and the solutions of systems of non-linear stochastic heat equations. The uniform Besov regularity of their corresponding local times is also obtained.
  
\end{abstract}

\maketitle

\section{Introduction and main results}
The local times of $d$-dimensional paths have gotten much interest during the last few decades by the analytical and probabilistic communities. The occupation measure basically measures the amount of time the path spends in a given set. The local time is defined as the Radon-Nikodym derivative
of the occupation measure with respect to the Lebesgue measure.
It is of importance in both theory and applications to investigate sample path properties of stochastic processes. One way of studying the irregularity properties of the sample paths of stochastic processes is by analyzing the smoothness of their local times. S. Berman has initiated this approach  in a series of papers \cite{Berman69a, Berman69b, Berman72} by using Fourier analytic methods to Gaussian processes. Furthermore, he has introduced the concept of local nondeterminism (LND) for Gaussian processes to investigate the existence of jointly continuous local times.
Since then, there have been a wide variety of extensions of the notion of local nondeterminism for Gaussian and stable processes, e.g. \cite{pitt1978local, cuzick1982joint, nolan1989local}. In the Gaussian case, the exponential form  of the characteristic function allows expressing the LND property in terms of a condition on the variance. Nevertheless, the unknown form of the characteristic functions of general processes leads to difficulties in extending the LND condition beyond the Gaussian framework. Consequently, the LND property used in the Gaussian context should be replaced, for general processes, by fine estimations on the characteristic function of the increments.
Recently, based on the conditional Malliavin calculus, Lou and Ouyang \cite{lou2017local} have established an upper bound of Gaussian type for
the partial derivatives of the $n$-point joint density of the solution to a stochastic
differential equation driven by fractional Brownian motion. They have used this result as
an alternative to the LND condition. Due to this, the authors in
\cite{lou2017local} have shown the existence and regularity of the local times of stochastic
differential equations driven by fractional Brownian motions. In \cite{boufoussi2021local}, a new condition, called $\alpha$-local nondeterminism ($\alpha$-LND for short), has been introduced to investigate the existence and joint regularity of the local times of the solutions to systems of non-linear stochastic heat equations---which are neither Gaussian nor stable processes. Although the $\alpha$-LND condition has been specifically designed  for non-Gaussian processes, in a Gaussian setting,  the $\alpha$-LND property can be seen as a weaker condition than the classical LND (see Remark \ref{rem alphha LND}(ii)).  Generally, the proof of the $\alpha$-LND condition relies on the technique of integration by parts derived from the Malliavin calculus.
 Roughly speaking, we believe that the approach presented in \cite{boufoussi2021local} can be used to establish the $\alpha$-LND for a class of adapted stochastic processes that are smooth in the Malliavin sense.

Consider $(X_t)_{t\in [0,1]}$ an $\R^d$-valued continuous stochastic process, such that  $X(0) = 0$. Assume that $X$ satisfies the $\alpha$-LND with $\alpha\in (0,1)$, see Definition \ref{alphha LND}.  Let us state the following hypothesis on $X$:\\
\textbf{H}\; There exists $p_0>\frac{1}{\alpha}$ and $K>0$ such that for all $0\leq s\leq t\leq 1$,
\begin{equation}\label{moment leq}
    \E[\|X_t-X_s\|^{p_0}]\leq K |t-s|^{p_0\alpha}.
  \end{equation}
One can get by the same calculations as in \cite{boufoussi2021local} (see \cite[Remark 5.6]{boufoussi2021local}) that, when $\alpha<\frac{1}{d}$, $X$ has a jointly continuous version of the local time, $L(x,t)$, satisfying almost surely a  H\"{o}lder condition of order $\gamma<1-d\alpha$, in the time variable $t$ uniformly in the space variable $x$. One of the main aims of this article is to improve this uniform H\"{o}lder continuity of $L(x,\bullet)$ to regularity in the Besov spaces---we refer to  Subsection \ref{subsection1} for some notions on Besov spaces.
There are well-known Besov regularity results, in the space variable $x$
for fixed $t$, of the local times $L(x, t)$ of some classical Gaussian and stable processes, e.g. \cite{boufoussi1993temps, boufoussi1999regularite}. However, to the best of our knowledge, there is no work in the literature treating the Besov regularity of $L(x,\bullet)$ even for the Gaussian or stable processes.
To fill this gap, we  investigate the  Besov regularity, in the time variable $t$ uniformly in the space variable $x$ and in $p$ (the index of the Besov space $\mathbf{B}^{\nu}_{p,\infty}(I;\R)$), for local times of continuous general processes that satisfy the $\alpha$-LND condition. Our first main result is:
\begin{thm}\label{thm Besov regularity of L}
  Let $X=(X_t)_{t\in [0,1]}$ be an $\R^d$-valued continuous stochastic process which is $\alpha$-LND with $\alpha\in (0,\frac{1}{d})$. Assume also that  $X$ verifies \normalfont{\textbf{H}}. Denote by $L(x,t)$ the jointly continuous version of the local time of $X$. Then, almost surely, for any $1\leq p<\infty$,
  \begin{equation}\label{modulus of continuity of L unif x}
    \sup_{0<t\leq 1}t^{-(1-d\alpha)}\sup_{|h|\leq t}\|r\mapsto \sup_{x\in \R^d}|L(x,r+h)-L(x,r)|\|_{L^p(I(h);\R)}<\infty.
  \end{equation}
In particular, we have the following Besov regularity
  \begin{equation}\label{Besov L}
    \mathbb{P}\left[L(x,\bullet)\in \mathbf{B}^{1-d\alpha}_{p,\infty}(I;\R)\,,\;\text{ for all } x\in \R^d \text{ and } p\in [1,\infty)\right]=1,
  \end{equation}
   where $I=[0,1]$, $I(h)=\{x\in I\,;\;x+h\in I\}$, and $L(x,\bullet):t\in I\mapsto L(x,t)$.
\end{thm}

As mentioned above, we know that the regularity of the local time is linked to the irregular behavior of the sample paths of the corresponding process.
Recall that Adler's theorem, \cite[Theorem 8.7.1]{adler1981geometry}, establishes a connection between the H\"{o}lder continuity, in the time variable $t$ uniformly on the space variable $x$, of the local time and the H\"{o}lder irregularity of its underlying function, as follows:
\begin{thm}[Adler]
  Let $(f(t))_{t\in [0,1]}$ be an $\R^d$-valued continuous function possessing a local time, $L(x,t)$, satisfying: there exist $c$ and $\rho$ positive and finite constants, such that  for all $t,t+h\in [0,1]$ and all $|h|<\rho$,
  \begin{equation}\label{Adler unifor Holder}
    \sup_{x\in \R^d}|L(x,t+h)-L(x,t)|\leq c|h|^{\mu},
  \end{equation}
for $0<\mu<1$.
Then,  all coordinate functions of $f$ are nowhere H\"{o}lder continuous of order greater than $(1-\mu)/d$.
\end{thm}

One of the goals of this article is to provide a similar theorem to that of Adler in the case of Besov spaces. We have the following theorem:
\begin{thm}\label{Adler Besov}
 Let $(f(t))_{t\in [0,1]}$ be an $\R^d$-valued continuous function possessing a local time, $L(x,t)$, satisfying for some $\mu\in (0,1)$ and $p\in (d/(1-\mu),\infty)$,
  \begin{equation}\label{modulus continuity Besov Adler}
 \sup_{0<t\leq 1} t^{-\mu/d}\sup_{|h|\leq t}\|s\mapsto \sup_{x\in \R^d}|L(x,s+h)-L(x,s)|^{\frac{1}{d}}\|_{L^{\frac{p}{(p-1)}}(I(h);\R)}<\infty.
\end{equation}
  Then,  the function $f$ does not belong to the Besov space $\mathbf{B}^{(1-\mu)/d,0}_{p,\infty}(I,\R^d)$, where $I=[0,1]$  and $I(h)=\{x\in I\,;\;x+h\in I\}$.
\end{thm}

Based on the above results, we note that the Besov regularity, in the time variable $t$ uniformly on the space variable $x$ and in $p$, of the local times $L(x,t)$ associated with $\alpha$-LND stochastic processes is valuable, as this knowledge can be applied towards uniform (in $p$) Besov irregularity of the underlying processes. Therefore, as a consequence of Theorem \ref{thm Besov regularity of L} and \ref{Adler Besov}, we will obtain the following theorem:
\begin{thm}\label{thm non Besov X}
  Let $(X_t)_{t\in [0,1]}$  be an $\R^d$-valued continuous stochastic process, $X(0)=0$, which is $\alpha$-LND with $\alpha\in (0,\frac{1}{d})$. Assume also that  $X$ verifies \normalfont{\textbf{H}}. Then
  \begin{equation}\label{No Besov}
    \mathbb{P}\left[X(\bullet)\in \mathbf{B}^{\alpha,0}_{p,\infty}(I,\R^d)\,,\;\text{ for some } p\in(1/\alpha,\infty) \right]=0,
  \end{equation}
  where $I=[0,1]$ and $X(\bullet):t\in I\mapsto X_t\in \R^d$.
\end{thm}
According to the following continuous injections
$$\mathbf{B}^{\alpha}_{p,q}(I,\R^d)\hookrightarrow \mathbf{B}^{\alpha,0}_{p,\infty}(I,\R^d),\qquad 1\leq q<\infty,$$
we get:
\begin{cor}\label{cor non Besov X p q}
   Let $(X_t)_{t\in [0,1]}$  be an $\R^d$-valued continuous stochastic process, $X(0)=0$, which is $\alpha$-LND with $\alpha\in (0,\frac{1}{d})$. Assume also that  $X$ verifies \normalfont{\textbf{H}}. Then
   \begin{equation}\label{No Besov p q}
    \mathbb{P}\left[X(\bullet)\in \mathbf{B}^{\alpha}_{p,q}(I,\R^d)\,,\;\text{ for some } p\in(1/\alpha,\infty) \text{ and } q\in [1,\infty)\right]=0,
  \end{equation}
  where $I=[0,1]$ and $X(\bullet):t\in I\mapsto X_t\in \R^d$.
\end{cor}
In the literature, there was a few separate attempts to study the Besov irregularity of some real valued Gaussian processes, we mention \cite{roynette1993mouvement, boufoussi2020besov1, boufoussi2020besov}. In contrast to the proof of Theorem \ref{thm non Besov X}, which relies  on local times, the proof of these  articles is based on a straightforward calculation.
In Section \ref{Examples}, we illustrate Theorem  \ref{thm non Besov X} and \ref{thm Besov regularity of L} with some examples. We infer the uniform Besov regularity of the local times, as well as the uniform Besov irregularity, of the following $\alpha$-LND processes: First, we consider an $\R^d$-valued Gaussian process $(Y_t)_{t\in [0,1]}$, with $Y_t=(Y^1,\ldots,Y^d_t)$, where $Y^1,\ldots,Y^d_t$ are independent copies of a real-valued LND Gaussian process $(Y^0_t)_{t\in [0,1]}$, satisfying for some $\alpha\in (0,1)$ and finite constants $c, C>0$,
  \begin{equation}\label{minoration var}
    c(t-s)^{2\alpha}\leq\var\left(Y^0_{t}-Y^0_{s}\right)\leq C(t-s)^{2\alpha},
  \end{equation}
  for every $0\leq s<t\leq 1$.
Notice that the $d$-dimensional bifractional Brownian motion verifies the above conditions.
Secondly, as an example of non-Gaussian and non-stable processes, we regard systems of non-linear stochastic heat equations.

In this article, to deal with Besov spaces, we are based on representations of the Besov norms in terms
of dyadic expansion coefficients of a given function. These descriptions of the Besov norms are derived from \cite[3.b.9 Corollary]{konig1986eigenvalue}. To the best of our knowledge, the treatise of K\"{o}nig \cite{konig1986eigenvalue} has been used  for the first time to investigate Besov regularity of stochastic processes by Hyt\"{o}nen and Veraar \cite{hytonen2008besov}.


The rest of the paper is organized as follows. In the second section, we write
some preliminary results on Besov spaces and local times. The third section is devoted to the proofs of the main results. In
the fourth section, we give some examples.

Finally, we point out that constants in our proofs may change from line to line. For a process $X=(X_t)_{t\in [0,1]}$, sometimes if necessary, we write $X(t)$ instead of  $X_t$.
\section{Preliminaries}
\subsection{Besov spaces}\label{subsection1}
For the definition of the real-valued Besov spaces, we refer to \cite{triebel1995interpolation}, and for the vector-valued Besov spaces, we suggest the treatise \cite{konig1986eigenvalue}. Let $I=[0,1]$, for any $h\in \R$, we put $I(h)=\{x\in I\,;\;x+h\in I\}$.
Let $1\leq p<\infty$ and $\nu\in (0,1)$, the modulus of continuity of a function $f\in L^p(I;\R^d)$ is defined by
\begin{equation}\label{modulus of continuity}
  \omega_p(f,t)=\sup_{|h|\leq t}\|x\mapsto f(x+h)-f(x)\|_{L^p(I(h);\R^d)}.
\end{equation}
We define the vector-valued Besov space $\mathbf{B}^{\nu}_{p,\infty}(I;\R^d)$ as the space of all functions $f\in L^p(I;\R^d)$ such that the seminorm $\mathcal{N}_{\nu,p}(f):=\sup_{0<t\leq 1}t^{-\nu}\omega_p(f,t)$ is finite. $\mathbf{B}^{\nu}_{p,\infty}(I;\R^d)$ endowed with the sum of the $L^p$-norm and the seminorm $\mathcal{N}_{\nu,p}$ is a Banach space. Let $\mathbf{B}^{\nu,0}_{p,\infty}(I,\R^d)$ be the space of all functions $f\in \mathbf{B}^{\nu}_{p,\infty}(I;\R^d)$ for which $\lim_{t\to 0^+}t^{-\nu}\omega_p(f,t)=0$. Using a dyadic approximation argument (see the lemma in page 173 and Corollary 3.b.9 in \cite{konig1986eigenvalue}) one has the following theorem:
\begin{thm}\label{thm equivalance norm}
  Let $1\leq p<\infty$ and $\nu\in (0,1)$. We have
  \begin{description}
    \item[(i)] The seminorm $\mathcal{N}_{\nu,p}$ is equivalent to
    \begin{equation}\label{discretisation}
      \|f\|_{\nu,p}:= \sup_{j\geq 0}2^{j\nu}\|x\mapsto f(x+2^{-j})-f(x)\|_{L^p(I(2^{-j});\R^d)}.
    \end{equation}
    \item[(ii)] Let $f\in L^p(I;\R^d)$. Then $f$ is in $\mathbf{B}^{\nu,0}_{p,\infty}(I,\R^d)$ if and only if
    \begin{equation}\label{discretisation B0}
      \lim_{j\to \infty} 2^{j\nu}\|x\mapsto f(x+2^{-j})-f(x)\|_{L^p(I(2^{-j});\R^d)}=0.
    \end{equation}
    \item[(iii)] Let $1\leq p<\infty$ and $0<\nu<1$. Then there exists a positive and finite constant $c$ such that for all $g$ from $\R^d\times [0,1]$ to $\R$ jointly continuous function with compact support,
        \begin{equation}\label{modulus uniform equvalence}
        \begin{split}
           &c^{-1}\sup_{0<t\leq 1}t^{-\nu}\sup_{|h|\leq t}\|r\mapsto \sup_{x\in \R^d}|g(x,r+h)-g(x,r)|\|_{L^p(I(h);\R)}  \\
           &\leq \sup_{j\geq 0} 2^{j\nu}\|r\mapsto \sup_{x\in \R^d}|g(x,r+2^{-j})-g(x,r)|\|_{L^p(I(2^{-j});\R)} \\
           & \leq c\sup_{0<t\leq 1}t^{-\nu}\sup_{|h|\leq t}\|r\mapsto \sup_{x\in \R^d}|g(x,r+h)-g(x,r)|\|_{L^p(I(h);\R)}.
           \end{split}
        \end{equation}
    \item[(iv)] Let $1\leq p<\infty$, $0<\nu<1$, and $f$ be an $\R^d$-valued continuous function. Then
$$\lim_{t\to 0^+}t^{-\nu}\sup_{|h|\leq t}\|x\mapsto\sup_{y,z\in[0,1]}\|f(hz+x)-f(hy+x)\|\|_{L^p(I(h),\R)}=0,$$
if and only if
$$\lim_{j\to \infty} 2^{j\nu}\|x\mapsto\sup_{y,z\in [0,1]}\| f(2^{-j}z+x)-f(2^{-j}y+x)\|\|_{L^p(I(2^{-j});\R)}=0.$$
  \end{description}
\end{thm}

  Now we will introduce the spaces $\mathbf{b}^{\nu,0}_{p,\infty}(I;\R^d)$, which will play a key role in the proof of Theorem \ref{Adler Besov}.  Let $I=[0,1]$, $1\leq p<\infty$, and $\nu\in (0,1)$, $\mathbf{b}^{\nu,0}_{p,\infty}(I;\R^d)$ is defined as the space of all functions $f\in C(I;\R^d)$, where $C(I;\R^d)$ is the space of $\R^d$-valued continuous functions, such that
  \begin{equation}\label{discretisation b0}
      \lim_{j\to \infty} 2^{j\nu}\|x\mapsto\sup_{y,z\in [0,1]}\| f(2^{-j}z+x)-f(2^{-j}y+x)\|\|_{L^p(I(2^{-j});\R)}=0,
    \end{equation}
here $I(2^{-j})=\{x\in I\,;\;x+2^{-j}\in I\}$.
In the below theorem, we will give the relation between the spaces $\mathbf{b}^{\nu,0}_{p,\infty}(I;\R^d)$ and the classical spaces $\mathbf{B}^{\nu,0}_{p,\infty}(I,\R^d)$.
\begin{thm}\label{B0 equaal b0}
  Let $I=[0,1]$, $\nu\in (0,1)$, and $\frac{1}{\nu}< p<\infty$. Then
  $$\mathbf{b}^{\nu,0}_{p,\infty}(I;\R^d)=C(I;\R^d)\cap \mathbf{B}^{\nu,0}_{p,\infty}(I,\R^d).$$
\end{thm}
At this point, to prove the above theorem, we need the Garsia-Rodemich-Rumsey inequality \cite{garsia1970real}.
\begin{lem}\label{GRR}
  Let $\Psi(u)$ and $p(u)$ be non-negative even functions respectively on $\R$ and $[-1,1]$ with $p(0)=0$ and $\Psi(\infty)=\infty$. Assume that $p(u)$ and $\Psi(u)$ are non decreasing for $u\geq0$ and $p(u)$ is continuous. Let $g(x)$ be continuous on $[0,1]$ and suppose that
$$\int_{0}^{1}\int_{0}^{1}\Psi\left(\frac{g(v)-g(w)}{p(v-w)}\right)dv\,dw\leq B<\infty.$$
Then, for all $z,y\in [0,1]$,
$$|g(z)-g(y)|\leq 8\int_{0}^{|z-y|}\Psi^{-1}\left(\frac{4B}{u^2}\right)dp(u).$$
\end{lem}
\begin{proof}[Proof of Theorem \ref{B0 equaal b0}]
  We just need to prove that $C(I;\R^d)\cap \mathbf{B}^{\nu,0}_{p,\infty}(I,\R^d)\subseteq \mathbf{b}^{\nu,0}_{p,\infty}(I;\R^d)$, since the second inclusion is trivial.
  Let $f=(f_1,\ldots,f_d)\in C(I;\R^d)\cap \mathbf{B}^{\nu,0}_{p,\infty}(I,\R^d)$, $\Psi(u)=|u|^p$, and $p(u)=|u|^{\nu+\frac{\beta}{p}}$, where $\beta\in [0,1)$ such that $\nu p-1>1-\beta$. Hence by Fubini's theorem, changes of variables, and the fact that $f\in\mathbf{B}^{\nu}_{p,\infty}(I,\R^d)$, we have for any $1\leq i\leq d$,
  \begin{equation}\label{Bijx}
    B(i,j,x):=\int_{0}^{1}\int_{0}^{1}\frac{|f_i(2^{-j}v+x)-f_i(2^{-j}w+x)|^p}{|v-w|^{\nu p+\beta}}dv\,dw<\infty.
  \end{equation}
  Therefore, by Lemma \ref{GRR} we get for all $1\leq i\leq d$ and $z,y\in [0,1]$,
  \begin{align*}
    &|f_i(2^{-j}z+x)-f_i(2^{-j}y+x)|^p \\
    &\leq C_{\nu,\beta,p} |z-y|^{\nu p+\beta-2} \int_{0}^{1}\int_{0}^{1}\frac{|f_i(2^{-j}v+x)-f_i(2^{-j}w+x)|^p}{|v-w|^{\nu p+\beta}}dv\,dw\\
    &\leq C_{\nu,\beta,p} \int_{0}^{1}\int_{0}^{1}\frac{|f_i(2^{-j}v+x)-f_i(2^{-j}w+x)|^p}{|v-w|^{\nu p+\beta}}dv\,dw,
  \end{align*}
  where $C_{\nu,\beta,p}= 8^p 4 p^p/(\beta+p\nu-2)^p$. Hence, for all $1\leq i\leq d$,
  \begin{equation*}
    \begin{split}
        &\sup_{z,y\in [0,1]}|f_i(2^{-j}z+x)-f_i(2^{-j}y+x)|^p \\
        &\leq C_{\nu,\beta,p} \int_{0}^{1}\int_{0}^{1}\frac{|f_i(2^{-j}v+x)-f_i(2^{-j}w+x)|^p}{|v-w|^{\nu p+\beta}}dv\,dw,
    \end{split}
  \end{equation*}
  Therefore, for all $1\leq i\leq d$,
  \begin{align}\label{app GRR}
        &2^{jp\nu}\int_{I(2^{-j})}\sup_{z,y\in [0,1]}|f_i(2^{-j}z+x)-f_i(2^{-j}y+x)|^p dx\nonumber\\
        &\leq C_{\nu,\beta,p} \int_{[0,1]^2}\frac{2^{jp\nu}}{|v-w|^{\nu p+\beta}}\int_{0}^{1-2^{-j}}|f_i(2^{-j}v+x)-f_i(2^{-j}w+x)|^p dx\,dv\,dw\nonumber\\
        &=C_{\nu,\beta,p} \int_{[0,1]^2}\frac{2^{jp\nu}}{|v-w|^{\nu p+\beta}}\int_{2^{-j}w}^{1-2^{-j}+2^{-j}w}|f_i(x+2^{-j}(v-w))-f_i(x)|^p dx\,dv\,dw\nonumber\\
        &\leq C_{\nu,\beta,p} \int_{[0,1]^2}\frac{2^{jp\nu}}{|v-w|^{\nu p+\beta}}\int_{I(2^{-j}(v-w))}|f_i(x+2^{-j}(v-w))-f_i(x)|^p dx\,dv\,dw,
  \end{align}
  where $I(2^{-j}(v-w))=\{x\in [0,1]\,;\;x+2^{-j}(v-w)\in [0,1]\}$. By the definition of $f\in \mathbf{B}^{\nu,0}_{p,\infty}(I,\R^d)$, we have for all $v,w\in [0,1]$ with $v\neq w$,
  $$\lim_{j\to \infty}\frac{2^{jp\nu}}{|v-w|^{\nu p+\beta}}\int_{I(2^{-j}(v-w))}|f_i(x+2^{-j}(v-w))-f_i(x)|^p dx=0.$$
  Therefore, By Lebesgue's dominated convergence theorem the right hand side of \eqref{app GRR} converges to $0$. Which concludes the proof of  Theorem \ref{B0 equaal b0}.
  \end{proof}
  \begin{rem}
    Let $\mathbf{B}^{\nu}_{p,q}(I,\R^d)$, for $1\leq p,q<\infty$ and $\nu\in (0,1)$, be the Besov spaces defined as follows:
    $$\mathbf{B}^{\nu}_{p,q}(I,\R^d):=\{f\in L^p(I,\R^d)\,;\;\|f\|_{\nu,p,q}<\infty\},$$
    where by \cite[3.b.9 Corollary]{konig1986eigenvalue} we have
    $$\|f\|_{\nu,p,q}:=\left\{\sum_{j\geq 0}2^{jq\nu}\|x\mapsto f(x+2^{-j})-f(x)\|^q_{L^p(I(2^{-j});\R^d)}\right\}^{\frac{1}{q}}.$$
    Therefore, we have the following continuous injection: for all $1\leq p,q<\infty$ and $\nu\in (0,1)$,
    \begin{equation}\label{injection Besov}
      \mathbf{B}^{\nu}_{p,q}(I,\R^d)\hookrightarrow \mathbf{B}^{\nu,0}_{p,\infty}(I,\R^d).
    \end{equation}
  \end{rem}
\subsection{The local times}\label{The local times}
This section is devoted to  give some aspects of the theory of local times. For more details on the subject, we refer to the survey of Geman and Horowitz  \cite{GemanHorowitz}.

Let $(\varphi_t)_{t\in [0,T]}$ be an $\R^d$-valued Borel function. For any Borel set $B\subseteq [0,T]$, the occupation measure of $\varphi$ on $B$ is given by the following measure on $\R^d$:
$$\nu_B(\bullet)=\lambda\{t\in B\,;\;\varphi_t\in \bullet\},$$
where $\lambda$ is the Lebesgue measure. When $\nu_B$ is absolutely continuous with respect to the Lebesgue measure on $\R^d$, $\lambda_d$, we say that the local time of $\varphi$ on $B$ exists and it is defined as the Radon-Nikodym derivative of $\nu_B$ with respect to $\lambda_d$, i.e., for almost every $x$,
$$L(x,B)=\frac{d\nu_B}{d\lambda_d}(x).$$
In the above, we call $B$ the time variable and $x$ the space variable. We write $L(x,t)$ and $L(x)$ instead of respectively  $L(x,[0,t])$ and $L(x,[0,T])$.

The local time fulfills the following occupation formula: for any Borel set $B\subseteq [0,T]$, and for every measurable bounded function $f:\R^d\to \R$,
\begin{equation}\label{occupation formula}
  \int_{B}f(\varphi_s)ds=\int_{\R^d}f(x)L(x,B)dx.
\end{equation}

The deterministic function $\varphi$ can be chosen to be the sample path of a separable stochastic process $(X_t)_{t\in [0,T]}$ with $X(0)=0$ a.s. In this regard, we state that the process $X$ has a local time (resp. square integrable local time) if for almost all $\omega$, the trajectory $t\mapsto X_t(\omega)$ has a local time (resp. square integrable local time).

We study the local time through Berman's approach. The idea is to derive properties of the local time, $L(\bullet,B)$, from the integrability properties of the Fourier transform of the sample paths of the process $X$.

Let us state the following hypotheses: \\
\textbf{A1}\;$$\int_{\R^d}\int_{0}^{T}\int_{0}^{T}\E\left[e^{i\left<u,X_t-X_s\right>}\right]dt\,ds\,du<\infty,$$
where $\left<\cdot,\cdot\right>$ is the Euclidean inner product on $\R^d$.\\
\textbf{A2}\; For every even integer $m\geq 2$,
$$\int_{(\R^d)^m}\int_{[0,T]^m}\left|\E\left[\exp\left(i\sum_{j=1}^{m}\left<u_j,X_{t_j}\right>\right)\right]\right|\prod_{j=1}^{m}dt_j\prod_{j=1}^{m}du_j<\infty.$$

Recall the following essential  result that we can find in \cite{Berman69a}:
\begin{thm}
  Assume \normalfont{\textbf{A1}}. Hence the process $X$ has a square integrable local time. Furthermore, we have almost surely, for all Borel set $B\subseteq [0,T]$, and for almost every $x$,
  \begin{equation}\label{local time L2}
    L(x,B)=\frac{1}{(2\pi)^d}\int_{\R^d}e^{-i\left<u,x\right>}\int_B e^{i\left<u,X_t\right>}dt\,du.
  \end{equation}
\end{thm}

Remark that $L(x,B)$, given by \eqref{local time L2}, is not a stochastic process. We will follow Berman \cite{Berman69b} to create  a version of the local time, which is a stochastic process.
The following theorem is given in Berman \cite[Theorem 4.1]{Berman69b} for $d=1$ and $m=2$, so we will omit its proof.
\begin{thm}
  Assume \normalfont{\textbf{A1}} and \normalfont{\textbf{A2}}. Put for all integer $N\geq 1$,
  $$L_N(x,t)=\frac{1}{(2\pi)^d}\int_{[-N,N]^d}e^{-i\left<u,x\right>}\int_{0}^{t} e^{i\left<u,X_s\right>}ds\,du.$$
  Therefore, there exists a stochastic process $\tilde{L}(x,t)$ separable in the $x$-variable, such that for each even integer $m\geq 2$,
  \begin{equation}\label{L tilde}
    \lim_{N\to \infty}\sup_{(x,t)\in \R^d\times [0,T]}\E\left[|L_N(x,t)-\tilde{L}(x,t)|^m\right]=0.
  \end{equation}
\end{thm}
\begin{thm}[Theorem 4.3 in \cite{Berman69b}]
  Let $\tilde{L}(x,t)$ be given by $\eqref{L tilde}$.
  If the stochastic process $\{\tilde{L}(x,t),\, x\in\R^d\}$ is almost surely continuous, hence it is a continuous (in the $x$-variable) version of the local time on $[0,t]$.
\end{thm}

In order to prove Theorem \ref{thm Besov regularity of L},  we will need to estimate the moments of the increments of $\tilde{L}$. For this end, we have by \eqref{L tilde}, for all $x,y\in \R^d$, $t, h\in [0,T]$ such that $t+h\in [0,T]$, and even integer $m\geq 2$,
\begin{align}\label{E Lxkth Lxth Lxkt Lxt}
        &\E[\tilde{L}(x+y,t+h)-\tilde{L}(x,t+h)-\tilde{L}(x+y,t)+\tilde{L}(x,t)]^m=\frac{1}{(2\pi)^{md}} \nonumber\\
         & \times\int_{(\R^d)^m}\int_{[t,t+h]^m}\prod_{j=1}^{m}\left(e^{-i\left<u_j,x+y\right>}-e^{-i\left<u_j,x\right>}\right)\E\left[e^{i\sum_{j=1}^{m}\left<u_j,X_{t_j}\right>}\right]\prod_{j=1}^{m}dt_j\prod_{j=1}^{m}du_j\nonumber\\
         &=\frac{1}{(2\pi)^{md}}\int_{(\R^d)^m}\int_{[t,t+h]^m}\prod_{j=1}^{m}\left(e^{-i\left<v_j-v_{j+1},x+y\right>}-e^{-i\left<v_j-v_{j+1},x\right>}\right)\nonumber\\
&\qquad\qquad\qquad\qquad\qquad\qquad\times\E\left[e^{i\sum_{j=1}^{m}\left<v_j,X_{t_j}-X_{t_{j-1}}\right>}\right]\prod_{j=1}^{m}dt_j\prod_{j=1}^{m}dv_j,
\end{align}
and
\begin{equation}\label{E Lxth Lxt}
    \begin{split}
        &\E[\tilde{L}(x,t+h)-\tilde{L}(x,t)]^m\\
        &=\frac{1}{(2\pi)^{md}} \int_{(\R^d)^m}\int_{[t,t+h]^m}e^{-i\sum_{j=1}^{m}\left<u_j,x\right>}\,\E\left[e^{i\sum_{j=1}^{m}\left<u_j,X_{t_j}\right>}\right]\prod_{j=1}^{m}dt_j\prod_{j=1}^{m}du_j\\
        &=\frac{1}{(2\pi)^{md}} \int_{(\R^d)^m}\int_{[t,t+h]^m}e^{-i\left<v_1,x\right>}\,\E\left[e^{i\sum_{j=1}^{m}\left<v_j,X_{t_j}-X_{t_{j-1}}\right>}\right]\prod_{j=1}^{m}dt_j\prod_{j=1}^{m}dv_j,
    \end{split}
\end{equation}
where $t_0=0$, and the last equality in \eqref{E Lxkth Lxth Lxkt Lxt} (resp. \eqref{E Lxth Lxt}) holds through suitable changes of variables as follows:
$$u_j=v_j-v_{j+1},\qquad j=1,\cdots,m,\qquad\text{with}\qquad v_{m+1}=0.$$

In order to estimate \eqref{E Lxkth Lxth Lxkt Lxt} and \eqref{E Lxth Lxt}, we need first to manage the following characteristic function $\E\left[e^{i\sum_{j=1}^{m}\left<v_j,X_{t_j}-X_{t_{j-1}}\right>}\right]$. Therefore, we will use the bellow  condition called $\alpha$-local nondeterminism ($\alpha$-LND), which was introduced for the first time in \cite{boufoussi2021local}.
\begin{defn}\label{alphha LND}
  Let $X=(X_t)_{t\in [0,T]}$ be an $\R^d$-valued stochastic process, $J$ a subinterval of $[0,T]$ and $\alpha\in (0,1)$ . $X$ is said to satisfy the $\alpha$-LND property on $J$, if for every non-negative integers $m\geq 2$, and $k_{j,l}$, for $j=1,\cdots,m$, $l=1,\cdots,d$, there exist positive constants $c$ and $\varepsilon$, both may depend on $m$ and $k_{j,l}$, such that
  \begin{equation}\label{Ineq alphha LND}
        \left|\E\left[e^{i\sum_{j=1}^{m}\left<v_j,X_{t_j}-X_{t_{j-1}}\right>}\right]\right|
         \leq\frac{c}{\prod_{j=1}^{m}\prod_{l=1}^{d}|v_{j,l}|^{k_{j,l}}(t_j-t_{j-1})^{\alpha k_{j,l}}},
  \end{equation}
  for all $v_j=(v_{j, l}\,;\;1\leq l\leq d)\in (\R\setminus\{0\})^{d}$, for $j=1,\cdots,m$, and for every ordered points $t_1<\cdots<t_m$ in $J$ with $t_m-t_1<\varepsilon$ and $t_0=0$.
\end{defn}
\begin{rem}\label{rem alphha LND}
\begin{description}
\item[(i)]  It is well-known that the concept of local nondeterminism in the Gaussian framework means that ``the value of the process at a given time point is relatively unpredictable based on a finite set of observations from the immediate past". In the Gaussian context, Berman uses conditional variance to express this. But unfortunately, he can't use the conditional variance outside the Gaussian case because in a general framework  the conditional variance is not deterministic. So, Berman has introduced the local $g$-nondeterminism concept for general processes by replacing the incremental variance, which is a measure of local unpredictability, by a measure of local
predictability, namely, the value of the incremental density function at the origin, see \cite[Definition 5.1]{Bermangeneral1983}. By the Fourier inversion theorem, it is easy to see that the condition in Definition \ref{alphha LND} implies  the local $g$-nondeterminism condition. On the other hand, Nolan has introduced the notion of \textit{characteristic function locally approximately independent increments} (see \cite[Definition 3.1]{nolan1989local}), which is equivalent in the Gaussian and stable context to the classical LND condition. The condition in Definition \ref{alphha LND} ($d=1$) is an extension of Nolan's notion by replacing the characteristic functions $\left|\E\left[e^{ic_m u_{j}(X(t_j)-X(t_{j-1}))}\right]\right|$ in the right-hand side of \cite[Ineq. (3.3)]{nolan1989local} by $c\,|u_{j}|^{-k_{j}}(t_j-t_{j-1})^{-\alpha k_{j}}$.
  \item[(ii)] Although the $\alpha$-LND has been specifically created for non-Gaussian processes, a large class of Gaussian processes possesses this property.  Let $Y^0=(Y^0_t)_{t\in [0,T]}$ be a real-valued centred Gaussian process satisfying the classical local nondeterminism (LND) property on $J\subseteq [0,T]$ (see \cite[Lemma 2.3]{Berman73}). Assume also that
  there exists a positive constant $K$, such that for every $s,t\in J$ with $s<t$,
  \begin{equation*}
    K(t-s)^{2\alpha}\leq\var\left(Y^0_{t}-Y^0_{s}\right).
  \end{equation*}
  Define $Y_t=(Y^1_t,\cdots,Y^d_t),$ where $Y^1,\cdots,Y^d$ are independent copies of $Y^0$. Then $Y$ is $\alpha$-LND on $J$.
  \item[(iii)] The question of whether or not the $\alpha$-LND is strictly weaker than the classical local
nondeterminism in the Gaussian framework is an open problem that is currently the subject of our  investigation. 

\end{description}
\end{rem}
\section{Proof of the main results}
Our aim in this section is to prove Theorem \ref{thm Besov regularity of L},  \ref{thm non Besov X}, and \ref{Adler Besov}. We will need first the following preliminary lemmas.
\begin{lem}\label{lem Veraar form norm Besov}
Let $1\leq p<\infty$, $0<\nu<1$, and $I=[0,1]$. Then, for all $\R$-valued  jointly continuous function $g$ with compact support defined on $\R^d\times [0,1]$,
        \begin{equation}\label{modulus uniform equvalence222}
        \begin{split}
           & \sup_{j\geq 0} 2^{jp\nu}\|r\mapsto \sup_{x\in \R^d}|g(x,r+2^{-j})-g(x,r)|\|^p_{L^p(I(2^{-j});\R)} \\
           & =\sup_{j\geq 0}2^{jp\nu-j} \int_{0}^{1} \sum_{k=1}^{2^j-1}\sup_{x\in \R^d}|g(x,2^{-j}(s+k))-g(x,2^{-j}(s+k-1))|^p ds.
           \end{split}
        \end{equation}
\end{lem}
\begin{proof}
  Denote
\begin{equation*}
  Z_{j}=2^{jp\nu}\|r\mapsto \sup_{x\in \R^d}|g(x,r+2^{-j})-g(x,r)|\|^p_{L^p(I(2^{-j});\R)},
\end{equation*}
here $I(2^{-j})=\{t\in [0,1]\,;\;t+2^{-j}\in [0,1]\}$. We have
\begin{align*}
   Z_{j} &= 2^{jp\nu}\int_{0}^{1-2^{-j}}\,\sup_{x\in \R^d}|g(x,r+2^{-j})-g(x,r)|^p dr\nonumber\\
   &= 2^{jp\nu}\sum_{k=1}^{2^j-1} \int_{(k-1)2^{-j}}^{k2^{-j}}\,\sup_{x\in \R^d}|g(x,r+2^{-j})-g(x,r)|^p dr \nonumber\\
   &=2^{jp\nu-j} \int_{0}^{1} \sum_{k=1}^{2^j-1}\sup_{x\in \R^d}|g(x,2^{-j}(s+k))-g(x,2^{-j}(s+k-1))|^p ds,
\end{align*}
where we have used the change of variables $s=2^j(r-2^{-j}(k-1))$. Which finishes the proof of Lemma \ref{lem Veraar form norm Besov}.
\end{proof}
\begin{lem}\label{lem chebechev local time}
Let $X=(X_t)_{t\in [0,1]}$ be an $\R^d$-valued continuous stochastic process which is $\alpha$-LND with $\alpha\in (0,\frac{1}{d})$.
 Denote by $L(x,t)$ the jointly continuous version of the local time of $X$, then
\begin{itemize}
  \item For any positive integer $q$, there exists a positive constant $K=K(q,d,\alpha)$ such that for all integer $j\geq 1$ and $x\in \R^d$,
  \begin{equation}\label{x plus X djk}
    \mathbb{P}\left[\sum_{k=1}^{2^j-1} |L(x+X(d_{j,k}),D_{j,k})|^q\geq 2^{-jq(1-d\alpha)+j} \right]\leq K 2^{-jd\alpha};
  \end{equation}
  \item Let $a$ and $q$ be positive integers and $0<\theta<\{(\frac{1}{\alpha}-d)/2\}\wedge 1$, then there exists a positive constant $K=K(q,a,d,\alpha,\theta)$, such that for all integers $j,h\geq 1$ and any $x,y\in \R^d$ and $\gamma>0$,
  \begin{equation}\label{x plus X djk y plus X djk}
     \mathbb{P}\left[\sum_{k=1}^{2^j-1} |A_{j,k,x,y}|^q\geq 2^{-jq(1-d\alpha-\theta\alpha)+j}
       \|x-y\|^{q\theta}2^{\gamma h} \right]\leq K 2^{-jd\alpha} 2^{-2^a\gamma h},
  \end{equation}
  \end{itemize}
  where $d_{j,k}=2^{-j}(k-1)$, $D_{j,k}=[2^{-j}(k-1),2^{-j}(k+1)]$, and
  \begin{equation}\label{Ajkx}
    A_{j,k,x,y}=L(x+X(d_{j,k}),D_{j,k})-L(y+X(d_{j,k}),D_{j,k}).
  \end{equation}
\end{lem}
\begin{proof}
  We only prove \eqref{x plus X djk y plus X djk}. We have by H\"{o}lder's inequality, for all positive integer $a$,
  \begin{align}\label{Holder l}
    \E\left[\left\{\sum_{k=1}^{2^j-1} |A_{j,k,x,y}|^q\right\}^{2^a}\right] &= \sum_{k_1,\ldots,k_{2^a}=1}^{2^j-1}\E\left[\prod_{i=1}^{2^a}|A_{j,k_i,x,y}|^q\right]\nonumber\\
    &\leq \left\{\sum_{k=1}^{2^j-1}\E\left[|A_{j,k,x,y}|^{q2^a}\right]^{1/2^a}\right\}^{2^a}.
  \end{align}
  On the other hand, let $Y_t = X(t)- X(d_{j,k})$. The occupation measure of $Y$ is just the occupation measure
of $X$ translated by the (random) constant $X(d_{j,k})$. Since the occupation measure of $X$
has a jointly continuous density, the occupation measure of $Y$ has also a jointly continuous density
given by $L_Y(x,t)=L_X(x+X(d_{j,k}),t)$. Put $m=q2^a$, hence by \eqref{Ajkx} and \eqref{E Lxkth Lxth Lxkt Lxt}  we obtain
  \begin{align}\label{Ajkxy equality}
     &\E\left[|A_{j,k,x,y}|^{m}\right] \nonumber \\
     &= E\left[|L(x+X(d_{j,k}),D_{j,k})-L(y+X(d_{j,k}),D_{j,k})|^{m}\right]\nonumber\\
     &= E\left[|L_Y(x,D_{j,k})-L_Y(y,D_{j,k})|^{m}\right]\nonumber\\
     &=\frac{1}{(2\pi)^{md}}\int_{(\R^d)^{m}}\int_{(D_{j,k})^{m}}\prod_{n=1}^{m}\left(e^{-i\left<v_n-v_{n+1},x\right>}-e^{-i\left<v_n-v_{n+1},y\right>}\right)\nonumber\\
     &\qquad\qquad\qquad\qquad\qquad\qquad\times\E\left[e^{i\sum_{n=1}^{m}\left<v_n,X_{t_n}-X_{t_{n-1}}\right>}\right]\prod_{n=1}^{m}dt_n\prod_{n=1}^{m}dv_n,
  \end{align}
  By \eqref{Ajkxy equality} and the elementary inequality $|1-e^{i\rho}|\leq 2^{1-\theta}|\rho|^{\theta}$ for any $0<\theta<1$ and $\rho\in \R$, we get
\begin{equation}\label{t t plus h xi xi plus y J}
        \E\left[|A_{j,k,x,y}|^{m}\right]
        \leq 2^{-m(d+\theta-1)}\pi^{-md}\|x-y\|^{m\theta}\mathcal{J}(m,\theta),
\end{equation}
where
\begin{equation*}
  \begin{split}
      &\mathcal{J}(m,\theta)  \\
       & =\int_{(D_{j,k})^m}\int_{(\R^d)^m}\prod_{n=1}^{m}\|v_n-v_{n+1}\|^{\theta}\left|\E\left[e^{i\sum_{n=1}^{m}\left<v_n,X_{t_n}-X_{t_{n-1}}\right>}\right]\right|\prod_{n=1}^{m}dv_n\prod_{n=1}^{m}dt_n.
  \end{split}
\end{equation*}
We replace the integration over the domain $(D_{j,k})^m$ by the integration over the
subset  $\Lambda_{j,k}=\{2^{-j}(k-1)\leq t_1<\cdots<t_m\leq 2^{-j}(k+1)\}$, hence we obtain
\begin{equation*}
  \begin{split}
      &\mathcal{J}(m,\theta)  \\
       & =m!\int_{\Lambda_{j,k}}\int_{(\R^d)^m}\prod_{n=1}^{m}\|v_n-v_{n+1}\|^{\theta}\left|\E\left[e^{i\sum_{n=1}^{m}\left<v_n,X_{t_n}-X_{t_{n-1}}\right>}\right]\right|\prod_{n=1}^{m}dv_n\prod_{n=1}^{m}dt_n,
  \end{split}
\end{equation*}
where $t_0=0$ and $v_{m+1}=0$. By the fact that $\|b-c\|^{\theta} \leq \|b\|^{\theta}+\|c\|^{\theta}$ for each $0<\theta<1$ and $b,c\in \R^d$, it follows that
\begin{equation}\label{}
  \prod_{n=1}^{m}\|v_n-v_{n+1}\|^{\theta} \leq  \prod_{n=1}^{m} \left(\|v_n\|^{\theta}+\|v_{n+1}\|^{\theta}\right).
\end{equation}
 Remark that the right side of this last inequality is at most equal to a finite sum of terms each of the
form $\prod_{n=1}^{m}\|v_n\|^{\epsilon_n \theta}$, where $\epsilon_n=0,1,$ or $2$ and $\sum_{n=1}^{m}\epsilon_n=m$. Hence
\begin{equation}\label{J without alpha LND}
  \begin{split}
      &\mathcal{J}(m,\theta)\leq m! \sum_{(\epsilon_1,\cdots,\epsilon_m)\in \{0,1,2\}^m}\int_{\Lambda_{j,k}}\int_{(\R^d)^m}\prod_{n=1}^{m}\|v_n\|^{\epsilon_n\theta}\\
       &\qquad\qquad\qquad\qquad\qquad\qquad\quad\times\left|\E\left[e^{i\sum_{n=1}^{m}\left<v_n,X_{t_n}-X_{t_{n-1}}\right>}\right]\right|\prod_{n=1}^{m}dv_n\prod_{n=1}^{m}dt_n.
  \end{split}
\end{equation}
On the other hand, by the $\alpha$-LND property of the process $X$, we get for every nonnegative integers $m\geq 2$, $k_{n,l}$, for $n=1,\cdots,m$ and $l=1,\cdots,d$, there exists a constant $c=c(m,k_{n,l})$ such that
  \begin{equation}\label{Ineq alphha LND for u 2}
        \left|\E\left[e^{i\sum_{n=1}^{m}\left<v_n,X_{t_n}-X_{t_{n-1}}\right>}\right]\right|
         \leq\frac{c}{\prod_{n=1}^{m}\prod_{l=1}^{d}|v_{n,l}|^{k_{n,l}}(t_n-t_{n-1})^{\alpha k_{n,l}}},
  \end{equation}
where $v_n=(v_{n,1},\cdots,v_{n,d})$. Put $I^n_1=[-1/(t_n-t_{n-1})^{\alpha},1/(t_n-t_{n-1})^{\alpha}]$ and $I^n_2=\R\setminus I^n_1$, Therefore
\begin{equation}\label{R d m}
  (\R^d)^m=\bigcup_{i_{n,l}\in \{1,2\}\atop n=1,\cdots,m; l=1,\cdots,d}\prod_{n=1}^{m}\prod_{l=1}^{d} I^n_{i_{n,l}}.
\end{equation}
Set, for $n=1,\cdots,m$ and $l=1,\cdots,d$,
$$k_{n,l}(i_{n,l})=\left\{
  \begin{array}{ll}
    0, & \text{if}\quad i_{n,l}=1; \\
    4, & \text{if}\quad i_{n,l}=2,
  \end{array}
\right.$$
Hence, by \eqref{J without alpha LND}-\eqref{R d m}, we obtain
\begin{equation}\label{}
  \begin{split}
      \mathcal{J}(m,\theta)&\leq m!c\sum_{i_{n,l}\in \{1,2\}\atop n=1,\cdots,m; l=1,\cdots,d} \sum_{(\epsilon_1,\cdots,\epsilon_m)\in \{0,1,2\}^m}\int_{\Lambda_{j,k}}\int_{\prod_{n=1}^{m}\prod_{l=1}^{d} I^n_{i_{n,l}}}\\
       & \times \frac{\prod_{n=1}^{m}\|v_n\|^{\epsilon_n\theta}}{\prod_{n=1}^{m}\prod_{l=1}^{d}|v_{n,l}|^{k_{n,l}(i_{n,l})}(t_n-t_{n-1})^{\alpha k_{n,l}(i_{n,l})}}\prod_{n=1}^{m}dv_n\prod_{n=1}^{m}dt_n.
  \end{split}
\end{equation}
We remark that
\begin{equation*}
   \prod_{n=1}^{m}\|v_n\|^{\epsilon_n\theta} \leq \prod_{n=1}^{m} \left(|v_{n,1}|^{\epsilon_n\theta}+\cdots+|v_{n,d}|^{\epsilon_n\theta}\right) =\sum_{l_1,\cdots,l_d\in \{1,\cdots,d\}}\prod_{n=1}^{m} |v_{n,l_n}|^{\epsilon_n\theta}.
\end{equation*}
Therefore
\begin{equation*}
  \begin{split}
      \mathcal{J}(m,\theta)&\leq m!c\sum_{l_1,\cdots,l_d\in \{1,\cdots,d\}}\sum_{i_{n,l}\in \{1,2\}\atop n=1,\cdots,m; l=1,\cdots,d} \sum_{(\epsilon_1,\cdots,\epsilon_m)\in \{0,1,2\}^m}\int_{\Lambda_{j,k}}\int_{\prod_{n=1}^{m}\prod_{l=1}^{d} I^n_{i_{n,l}}} \\
       &\times \frac{\prod_{n=1}^{m} |v_{n,l_n}|^{\epsilon_n\theta}}{\prod_{n=1}^{m}\prod_{l=1}^{d}|v_{n,l}|^{k_{n,l}(i_{n,l})}(t_n-t_{n-1})^{\alpha k_{n,l}(i_{n,l})}}\prod_{n=1}^{m}dv_n\prod_{n=1}^{m}dt_n.
  \end{split}
\end{equation*}
By Fubini's theorem, the right side of the above expression is equal to
\begin{equation*}
  \begin{split}
      & \quad\qquad m!c\sum_{l_1,\cdots,l_d\in \{1,\cdots,d\}}\sum_{i_{n,l}\in \{1,2\}\atop n=1,\cdots,m; l=1,\cdots,d} \sum_{(\epsilon_1,\cdots,\epsilon_m)\in \{0,1,2\}^m}\int_{\Lambda_{j,k}}\prod_{n=1}^{m}\int_{\prod_{l=1}^{d} I^n_{i_{n,l}}} \\
       &\qquad\qquad\qquad\qquad\times \frac{ |v_{n,l_n}|^{\epsilon_n\theta}}{\prod_{l=1}^{d}|v_{n,l}|^{k_{n,l}(i_{n,l})}(t_n-t_{n-1})^{\alpha k_{n,l}(i_{n,l})}} dv_n\prod_{n=1}^{m}dt_n.
  \end{split}
\end{equation*}
\begin{equation}\label{J m theta prod}
  \begin{split}
      &= m!c\sum_{l_1,\cdots,l_d\in \{1,\cdots,d\}}\sum_{i_{n,l}\in \{1,2\}\atop n=1,\cdots,m; l=1,\cdots,d} \sum_{(\epsilon_1,\cdots,\epsilon_m)\in \{0,1,2\}^m}\\
      &\qquad\times\int_{\Lambda_{j,k}}\prod_{n=1}^{m}\prod_{l=1\atop l\neq l_n}^{d}\int_{ I^n_{i_{n,l}}} \frac{ 1}{|v_{n,l}|^{k_{n,l}(i_{n,l})}(t_n-t_{n-1})^{\alpha k_{n,l}(i_{n,l})}} dv_{n,l}\\
      &\qquad\qquad\qquad\times\int_{ I^n_{i_{n,l_n}}} \frac{1}{|v_{n,l_n}|^{k_{n,l_n}(i_{n,l_n})-\epsilon_n\theta}(t_n-t_{n-1})^{\alpha k_{n,l_n}(i_{n,l_n})}} dv_{n,l_n}\prod_{n=1}^{m}dt_n.
  \end{split}
\end{equation}
\begin{itemize}
  \item If $i_{n,l}=1$ or $2$ with $l\neq l_n$, then we have
  $$\int_{ I^n_{i_{n,l}}} \frac{ 1}{|v_{n,l}|^{k_{n,l}(i_{n,l})}(t_n-t_{n-1})^{\alpha k_{n,l}(i_{n,l})}} dv_{n,l}=\frac{K_1}{(t_n-t_{n-1})^{\alpha}},$$
  where the constant $K_1$ depends only on $i_{n,l}$.
  \item If $i_{n,l_n}=1$ or $2$, then we get
  $$\int_{ I^n_{i_{n,l_n}}} \frac{1}{|v_{n,l_n}|^{k_{n,l_n}(i_{n,l_n})-\epsilon_n\theta}(t_n-t_{n-1})^{\alpha k_{n,l_n}(i_{n,l_n})}} dv_{n,l_n}=\frac{K_2}{(t_n-t_{n-1})^{\alpha(1+\epsilon_n\theta)}},$$
  where the constant $K_2$ depends on $i_{n,l_n}$, $\theta$, and $\epsilon_n$ such that $\sup_{\theta,\epsilon_n}K_2<\infty$.
\end{itemize}
Combining the above discussion with \eqref{J m theta prod}, we obtain
\begin{align}\label{J prod Ehm}
     \mathcal{J}(m,\theta) &\leq  m!c_1\sum_{l_1,\cdots,l_d\in \{1,\cdots,d\}}\sum_{i_{n,l}\in \{1,2\}\atop n=1,\cdots,m; l=1,\cdots,d} \sum_{(\epsilon_1,\cdots,\epsilon_m)\in \{0,1,2\}^m}\nonumber\\
       &\qquad\qquad\qquad\qquad\times\int_{\Lambda_{j,k}}\prod_{n=1}^{m}\frac{1}{(t_n-t_{n-1})^{\alpha(d+\epsilon_n\theta)}}\prod_{n=1}^{m}dt_n\nonumber\\
       &= m!c_2 \sum_{(\epsilon_1,\cdots,\epsilon_m)\in \{0,1,2\}^m}\int_{\Lambda_{j,k}}\prod_{n=1}^{m}\frac{1}{(t_n-t_{n-1})^{\alpha(d+\epsilon_n\theta)}}\prod_{n=1}^{m}dt_n\nonumber\\
       &= m!c_2 \sum_{(\epsilon_1,\cdots,\epsilon_m)\in \{0,1,2\}^m}\int_{a_{j,k}}^{b_{j,k}}dt_1\int_{t_1}^{b_{j,k}}dt_2\cdots\int_{t_{m-1}}^{b_{j,k}}dt_m\nonumber\\
       &\qquad\qquad\qquad\qquad\qquad\qquad\qquad\times\prod_{n=1}^{m}\frac{1}{(t_n-t_{n-1})^{\alpha(d+\epsilon_n\theta)}},
\end{align}
where $a_{j,k}=2^{-j}(k-1)$ and $b_{j,k}=2^{-j}(k+1)$. Let $0<\theta<\left\{(\frac{1}{\alpha}-d)/2\right\}\wedge 1$. We integrate in the order of $dt_{m}, dt_{m-1},\ldots, dt_1$, and use changes of variables in each step to construct Beta functions. Hence,  \eqref{J prod Ehm} is equal to
\begin{equation}\label{int ajk bjk}
 \sum_{(\epsilon_1,\cdots,\epsilon_m)\in \{0,1,2\}^m} K_{\epsilon}\int_{a_{j,k}}^{b_{j,k}} (b_{j,k}-t_1)^{(m-1)(1-d\alpha)-\alpha\theta\sum_{i=0}^{m-2}\epsilon_{m-i}}  t_1^{-\alpha(d+\epsilon_1\theta)} dt_1,
\end{equation}
where $K_{\epsilon}=m!c_2\frac{1}{1-\alpha(d+\epsilon_m\theta)}\frac{\Gamma(2-\alpha(d+\epsilon_{m}\theta))\prod_{i=1}^{m-2}\Gamma(1-\alpha(d+\epsilon_{m-i}\theta))}{\Gamma(1+(m-1)(1-d\alpha)-\theta\sum_{i=0}^{m-2}\epsilon_{m-i})}$
and $\epsilon=(\epsilon_1,\cdots,\epsilon_m)$. Recall that $\sum_{n=1}^{m}\epsilon_n=m$. Then

\begin{itemize}
  \item If $k=1$. According to \eqref{t t plus h xi xi plus y J} and \eqref{int ajk bjk} we get
  \begin{equation}\label{k equal 1}
        \E\left[|A_{j,1,x,y}|^{m}\right]
        \leq K_1 \|x-y\|^{m\theta}2^{-jm(1-d\alpha -\alpha\theta)},
\end{equation}
  where $K_1$ is equal to $c\sum_{\epsilon\in \{0,1,2\}^m} K_{\epsilon}\frac{\Gamma(1+(m-1)(1-d\alpha)-\alpha\theta\sum_{i=0}^{m-2}\epsilon_{m-i})\Gamma(1-\alpha(d+\epsilon_1\theta))}{\Gamma(1+m(1-d\alpha-\alpha\theta))}$
  with $c=2^{-m(d+\theta-1)}\pi^{-md} 2^{m(1-d\alpha-\alpha\theta)}$.
  \item If $2\leq k\leq 2^j-1$. Therefore \eqref{int ajk bjk} is less than or equal to
  \begin{align}\label{k biger than 2}
 &\sum_{(\epsilon_1,\cdots,\epsilon_m)\in \{0,1,2\}^m} K_{\epsilon} (b_{j,k}-a_{j,k})^{(m-1)(1-d\alpha)-\alpha\theta\sum_{i=0}^{m-2}\epsilon_{m-i}} \int_{a_{j,k}}^{b_{j,k}} t_1^{-\alpha(d+\epsilon_1\theta)} dt_1  \nonumber\\
   & =\sum_{(\epsilon_1,\cdots,\epsilon_m)\in \{0,1,2\}^m} K_{\epsilon}2^{(m-1)(1-d\alpha)-\alpha\theta\sum_{i=0}^{m-2}\epsilon_{m-i}} 2^{-jm(1-d\alpha-\alpha\theta)} \nonumber\\
   &\qquad\qquad\qquad\qquad\qquad\qquad\qquad\qquad\qquad\times\int_{k-1}^{k+1} u^{-\alpha(d+\epsilon_1\theta)} du \nonumber\\
   & \leq \tilde{K}_2 2^{-jm(1-d\alpha-\alpha\theta)}\int_{k-1}^{k+1} u^{-\alpha d} du\leq \tilde{K}_2 2^{-jm(1-d\alpha-\alpha\theta)} (k-1)^{-\alpha d}.
\end{align}
Hence by \eqref{t t plus h xi xi plus y J} and \eqref{k biger than 2}, we have
\begin{equation}\label{Ajkxy k biger than 2}
        \E\left[|A_{j,k,x,y}|^{m}\right]
        \leq K_2 \|x-y\|^{m\theta}2^{-jm(1-d\alpha -\alpha\theta)}(k-1)^{-\alpha d}.
\end{equation}
\end{itemize}
Now we return to estimate $\E\left[\left\{\sum_{k=1}^{2^j-1} |A_{j,k,x,y}|^q\right\}^{2^a}\right]$. Recall that $m=q2^a$. According to \eqref{Holder l} and the convexity of the function $x\mapsto x^{2^a}$, we have
\begin{align*}
    &\E\left[\left\{\sum_{k=1}^{2^j-1} |A_{j,k,x,y}|^q\right\}^{2^a}\right]\\
    &\leq \left\{\sum_{k=1}^{2^j-1}\E\left[|A_{j,k,x,y}|^{q2^a}\right]^{1/2^a}\right\}^{2^a}\\
    &\leq 2^{a-1}\left(\E\left[|A_{j,1,x,y}|^{q2^a}\right]+\left\{\sum_{k=2}^{2^j-1}\E\left[|A_{j,k,x,y}|^{q2^a}\right]^{1/2^a}\right\}^{2^a}\right).
    \end{align*}
    Hence, using \eqref{k equal 1} and \eqref{Ajkxy k biger than 2}, we derive that this last term is less than or equal to
    \begin{align*}
    & \tilde{K}\|x-y\|^{q2^a\theta}2^{-jq2^a(1-d\alpha -\alpha\theta)}\left(1+\left\{\sum_{k=2}^{2^j-1}(k-1)^{-\alpha d/2^a}\right\}^{2^a}\right)\nonumber\\
    &\leq\tilde{K}\|x-y\|^{q2^a\theta}2^{-jq2^a(1-d\alpha -\alpha\theta)}\left(1+\left\{\sum_{k=2}^{2^j-1}2\int_{k-\frac{3}{2}}^{k-1}x^{-d\alpha/2^a}dx\right\}^{2^a}\right)\\
    &\leq\tilde{K}\|x-y\|^{q2^a\theta}2^{-jq2^a(1-d\alpha -\alpha\theta)}\left(1+\left\{2\int_{\frac{1}{2}}^{2^j-2}x^{-d\alpha/2^a}dx\right\}^{2^a}\right)\\
    &\leq \hat{K} \|x-y\|^{q2^a\theta}2^{-jq2^a(1-d\alpha -\alpha\theta)}\left(1+2^{2^aj-jd\alpha}\right).
  \end{align*}
Therefore
\begin{equation}\label{Ajxy final}
  \E\left[\left\{\sum_{k=1}^{2^j-1} |A_{j,k,x,y}|^q\right\}^{2^a}\right]\leq K\|x-y\|^{q2^a\theta}2^{-jq2^a(1-d\alpha -\alpha\theta)}2^{j(2^a-\alpha d)}.
\end{equation}
The remainder of the proof is by Chebyshev's inequality, i.e.
\begin{equation*}
  \begin{split}
      & \mathbb{P}\left[\sum_{k=1}^{2^j-1} |A_{j,k,x,y}|^q\geq 2^{-jq(1-d\alpha-\theta\alpha)+j}\|x-y\|^{q\theta}2^{\gamma h} \right] \\
       & \leq
2^{j2^aq(1-d\alpha-\theta\alpha)-2^aj}\|x-y\|^{-q2^a\theta}2^{-2^a\gamma h}\E\left[\left\{\sum_{k=1}^{2^j-1} |A_{j,k,x,y}|^q\right\}^{2^a}\right].
  \end{split}
\end{equation*}
Hence \eqref{Ajxy final} concludes the proof of Lemma \ref{lem chebechev local time}.
\end{proof}
\begin{lem}\label{lem estimation X}
  Let $(X_t)_{t\in [0,1]}$  be an $\R^d$-valued continuous stochastic process that verifies \normalfont{\textbf{H}}.
  Then, for all $0<\delta<1$, almost surely there exists $j_1=j_1(\omega,\delta)$ such that
  \begin{equation}\label{sup Djk}
    \sup_{1\leq k\leq 2^j-1}\sup_{t\in D_{j,k}}\|X(t)-X(d_{j,k})\|\leq 2^{-j(\alpha-\frac{2-\delta}{p_0})} \qquad \text{for } j\geq j_1,
  \end{equation}
  where $d_{j,k}$ and $D_{j,k}$ are as in Lemma \ref{lem chebechev local time}.
\end{lem}
\begin{proof}
  According to  \eqref{moment leq}, we have almost surely
  \begin{equation}\label{B finit}
    B:=\int_{0}^{1}\int_{0}^{1}\frac{\|X(t)-X(s)\|^{p_0}}{|t-s|^{\alpha p_0+\gamma}}dt\,ds<\infty.
  \end{equation}
  Then by Lemma \ref{GRR} with $\Psi(u)=|u|^{p_0}$ and $p(u)=|u|^{\alpha+\frac{\gamma}{{p_0}}}$ where $\frac{1}{{p_0}}<\alpha$ and $2-\alpha {p_0}<\gamma<1$, we derive
  \begin{equation}\label{sup net}
    \sup_{1\leq k\leq 2^j-1}\sup_{t\in D_{j,k}}\|X(t)-X(d_{j,k})\|^{p_0}\leq C_{\alpha,{p_0},\gamma}B2^{-j(\alpha {p_0}+\gamma-2)}.
  \end{equation}
    Therefore, by \eqref{moment leq}, \eqref{B finit}, and \eqref{sup net} we get
  \begin{equation}\label{moment sup X}
    \E\left[\sup_{1\leq k\leq 2^j-1}\sup_{t\in D_{j,k}}\|X(t)-X(d_{j,k})\|^{p_0}\right]\leq K_{\alpha,{p_0},\gamma}2^{-j(\alpha {p_0}+\gamma-2)}.
  \end{equation}
  Hence, let $0<\delta<1$ and $2-\alpha {p_0}<\gamma<1$ such that $\delta<\gamma$, then by Chebyshev's inequality and \eqref{moment sup X} we write
  \begin{equation}\label{Chebyshev X}
    \mathbb{P}\left[\sup_{1\leq k\leq 2^j-1}\sup_{t\in D_{j,k}}\|X(t)-X(d_{j,k})\|\geq 2^{-j(\alpha-\frac{2-\delta}{{p_0}})}\right]\leq K_{\alpha,{p_0},\gamma}2^{-j(\gamma-\delta)}.
  \end{equation}
  Therefore the Borel–Cantelli lemma concludes the proof of Lemma \ref{lem estimation X}.
\end{proof}
\begin{lem}\label{lem sup L Besov}
Let $(X_t)_{t\in [0,1]}$  be an $\R^d$-valued continuous stochastic process which is $\alpha$-LND with $\alpha\in (0,\frac{1}{d})$. Assume also that $X$ verifies \normalfont{\textbf{H}}.
Then for all positive integer $q$, almost surely we have
\begin{equation}\label{sup local time Besov}
  \sup_{j\geq 1}2^{jq(1-d\alpha)-j} \int_{0}^{1} \sum_{k=1}^{2^j-1}\sup_{x\in \R^d}|L(x,2^{-j}(s+k))-L(x,2^{-j}(s+k-1))|^q ds<\infty.
\end{equation}
\end{lem}
\begin{proof}
  Let $d_{j,k}=2^{-j}(k-1)$; $D_{j,k}=[2^{-j}(k-1),2^{-j}(k+1)]$ for $j\geq 1$ and $1\leq k\leq 2^j-1$. It follows from Lemma \ref{lem estimation X} that for all $2-\alpha {p_0}<\delta<1$, almost surely there exists $j_1=j_1(\omega,\delta)$ such that
  \begin{equation}\label{sup Djk 2}
    \sup_{1\leq k\leq 2^j-1}\sup_{t\in D_{j,k}}\|X(t)-X(d_{j,k})\|\leq 2^{-j(\alpha-\frac{2-\delta}{{p_0}})} \qquad \text{for } j\geq j_1,
  \end{equation}
  Let $\beta_j=2^{-j\alpha}$ and
  $$G_j=\left\{x\in \R^d\,; \|x\|\leq 2^{-j(\alpha-\frac{2-\delta}{{p_0}})}, \quad x=\beta_jb \quad\text{for some}\; b\in \mathbb{Z}^d\right\},$$
  where $\mathbb{Z}$ is the set of integers. The cardinality of $G_j$ verifies
  \begin{equation}\label{card Gj}
    \# G_j\leq \left(2\left[2^{j(\frac{2-\delta}{{p_0}})}\right]+1\right)^d\leq 3^d 2^{jd(\frac{2-\delta}{{p_0}})},
  \end{equation}
  where $\left[2^{j(\frac{2-\delta}{{p_0}})}\right]$ is the integral part of $2^{j(\frac{2-\delta}{{p_0}})}$. It follows from \eqref{card Gj} and  Lemma \ref{lem chebechev local time} that
  \begin{align}\label{}
     &\mathbb{P}\left[\sum_{k=1}^{2^j-1} |L(x+X(d_{j,k}),D_{j,k})|^q\geq 2^{-jq(1-d\alpha)+j} \quad\text{for some}\; x\in G_j\right]\nonumber\\
     &\leq \# G_j K 2^{-jd\alpha}\leq 3^d K 2^{-jd(\alpha -\frac{2-\delta}{{p_0}})}.
  \end{align}
  Since $2-\alpha {p_0}<\delta<1$, it follows by the Borel-Cantelli lemma that  almost
surely there exists $j_2=j_2(\omega,\delta)$ such that
  \begin{equation}\label{Lx}
    \sum_{k=1}^{2^j-1} \sup_{x\in G_j}|L(x+X(d_{j,k}),D_{j,k})|^q\leq 2^{-jq(1-d\alpha)+j}\quad\text{for}\; j\geq j_2.
  \end{equation}
For any fixed integers $j,h\geq 1$ and any $x\in G_j$, define
  \begin{equation}\label{F set}
    F(j,h,x)=\left\{y\in \R^d\,;\;y=x+\beta_j\sum_{n=1}^{h}\epsilon_n2^{-n}\quad\text{for}\;\epsilon_n\in\{0,1\}^d\right\}.
  \end{equation}
  A pair of points $y_1, y_2 \in F(j,h,x)$ is said to be linked if $y_2-y_1=\beta_j\epsilon 2^{-h}$ for
some $\epsilon\in\{0,1\}^d$. Then by \eqref{x plus X djk y plus X djk} with $\gamma$ and $a$ such that $\gamma<q\theta$ and $2^a>\frac{d}{\gamma}$, we have
  \begin{align}\label{x plus X djk y plus X djk2}
      & \mathbb{P}\left[\sum_{k=1}^{2^j-1} |A_{j,k,y_1,y_2}|^q\geq 2^{-jq(1-d\alpha-\theta\alpha)+j}\|y_1-y_2\|^{q\theta}2^{\gamma h} \right. \nonumber\\
       & \left. \qquad\text{for some } x\in G_j, h\geq 1 \text{ and some linked pair } y_1, y_2 \in F(j,h,x)\vphantom{\int}\right]\nonumber\\
     &\leq \#G_j \sum_{h=1}^{\infty}2^{hd} K 2^{-jd\alpha} 2^{-2^a\gamma h}\leq K 3^d 2^{-jd(\alpha-\frac{2-\delta}{{p_0}})}\sum_{h=1}^{\infty}2^{-h(2^a\gamma-d)}.\nonumber
  \end{align}
  As $2-\alpha {p_0}<\delta<1$ and $2^a>\frac{d}{\gamma}$, it follows by the Borel-Cantelli lemma that  almost surely there exists $j_3=j_3(\omega,\delta,\gamma)$ such that for $j\geq j_3$
  \begin{equation}\label{Ly1 moins Ly2}
  \begin{split}
     & \sum_{k=1}^{2^j-1} |L(y_1+X(d_{j,k}),D_{j,k})-L(y_2+X(d_{j,k}),D_{j,k})|^q  \\
       & \leq 2^{-jq(1-d\alpha-\theta\alpha)+j}\|y_1-y_2\|^{q\theta}2^{\gamma h},
  \end{split}
  \end{equation}
  for all $x\in G_j$, $h\geq 1$ and any linked pair $y_1, y_2 \in F(j,h,x)$. Let $\Omega_0$ be the
event that \eqref{sup Djk 2}, \eqref{Lx} and \eqref{Ly1 moins Ly2} hold, hence $\mathbb{P}[\Omega_0]=1$. Let $j\geq j_4:=\max\{j_1,j_2,j_3\}$ be fixed.
 For any $y\in \R^d$ with $\|y\|\leq 2^{-j(\alpha-\frac{2-\delta}{{p_0}})}$, we
represent $y$ in the form $y=\lim_{h\to\infty}y_h$, where
  $$y_h=x+\beta_j\sum_{n=1}^{h}\epsilon_n2^{-n}\quad(y_0=x,\;\epsilon_n\in\{0,1\}^d),$$
  for some $x\in G_j$.  Then each pair $y_{h-1},y_h$ is linked, so by \eqref{Ly1 moins Ly2}  and the continuity of $L(\cdot,D_{j,k})$ we have
  \begin{align*}
     &\sum_{k=1}^{2^j-1} |L(y+X(d_{j,k}),D_{j,k})-L(x+X(d_{j,k}),D_{j,k})|^q  \\
     &\leq   2^{-jq(1-d\alpha-\theta\alpha)+j}\sum_{h=1}^{\infty}|\beta_j2^{-h}|^{q\theta}2^{\gamma h}\\
     & =2^{-jq(1-d\alpha)+j}\sum_{h=1}^{\infty}2^{-h(q\theta-\gamma)}
  \end{align*}
  Since $\gamma<q\theta$, we have almost surely for $j\geq j_4$,
  \begin{equation}\label{Ly mois Lx}
    \sum_{k=1}^{2^j-1} |L(y+X(d_{j,k}),D_{j,k})-L(x+X(d_{j,k}),D_{j,k})|^q\leq C 2^{-jq(1-d\alpha)+j}.
  \end{equation}
  for all $y\in \R^d$ with $\|y\|\leq 2^{-j(\alpha-\frac{2-\delta}{{p_0}})}$. It follows from \eqref{Lx} and \eqref{Ly mois Lx} that almost surely for $j\geq j_4$,
  \begin{equation}\label{Ly}
    \sum_{k=1}^{2^j-1} |L(y+X(d_{j,k}),D_{j,k})|^q\leq C_1 2^{-jq(1-d\alpha)+j},
  \end{equation}
 for all $y\in \R^d$ with $\|y\|\leq 2^{-j(\alpha-\frac{2-\delta}{{p_0}})}$. On the other hand, we have almost surely for $j\geq j_4$,
  \begin{align*}
     &\int_{0}^{1} \sum_{k=1}^{2^j-1}\sup_{x\in \R^d}|L(x,2^{-j}(s+k))-L(x,2^{-j}(s+k-1))|^q ds  \\
     &\leq \sum_{k=1}^{2^j-1}\sup_{x\in \R^d}|L(x,D_{j,k})|^q.
     \end{align*}
     This last term is equal to
     \begin{align*}
     \sum_{k=1}^{2^j-1}\sup_{x\in X(D_{j,k})}|L(x,D_{j,k})|^q &\leq  \sum_{k=1}^{2^j-1}\sup_{y\in V}|L(y+X(d_{j,k}),D_{j,k})|^q\\
     &\leq C_1 2^{-jq(1-d\alpha)+j},
  \end{align*}
  where $V=\{y\in \R^d\,;\;\|y\|\leq 2^{-j(\alpha-\frac{2-\delta}{{p_0}})}\}$. This completes the proof of Lemma \ref{lem sup L Besov}.
\end{proof}
\begin{proof}[Proof of Theorem \ref{thm Besov regularity of L}]
According to Lemma \ref{lem sup L Besov} and \ref{lem Veraar form norm Besov}, and Theorem \ref{thm equivalance norm}(iii) we conclude the proof.
  \end{proof}
  Now we provide the proof of Theorem \ref{Adler Besov},  which clearly explains that if a
  functions local time, $L(x,t)$, is Besov regular, in $t$ uniformly in $x$, then this has a significant effect on the Besov irregularity of the function itself.
\begin{proof}[Proof of Theorem \ref{Adler Besov}]
  According to the occupation formula \eqref{occupation formula}, we have for all $t\in (0,1]$, $0<h\leq t$, and $s\in [0,1-h]$,
  \begin{align*}
    h &= \int_{f([s,s+h])}L(x,[s,s+h])dx\nonumber\\
           &\leq \lambda_d\left(f([s,s+h])\right)\sup_{x\in \R^d}L(x,[s,s+h])\nonumber\\
           &\leq \sup_{r,\tau\in[s,s+h]}\|f(r)-f(\tau)\|^d\sup_{x\in \R^d}L(x,[s,s+h])\nonumber\\
           &=\sup_{r,\tau\in[0,1]}\|f(hr+s)-f(h\tau+s)\|^d\sup_{x\in \R^d}L(x,[s,s+h]).
  \end{align*}
  Hence
  $$h^{1/d}\leq \sup_{r,\tau\in[0,1]}\|f(hr+s)-f(h\tau+s)\|\sup_{x\in \R^d}|L(x,[s,s+h])|^{1/d}.$$
  Therefore by H\"{o}lder's inequality we derive that  for all $t\in (0,1]$ and $0<h\leq t$,
  \begin{align}\label{X sup L}
    &(1-h)h^{1/d} \nonumber\\
    &\leq \int_{0}^{1-h}\sup_{r,\tau\in[0,1]}\|f(hr+s)-f(h\tau+s)\|\sup_{x\in \R^d}|L(x,[s,s+h])|^{1/d}ds  \nonumber\\
     &\leq \|s\mapsto\sup_{r,\tau\in[0,1]}\|f(hr+s)-f(h\tau+s)\|\|_{L^p(I(h),\R)}  \nonumber\\
     & \qquad\qquad\qquad\qquad\times \left\{\int_{0}^{1-h}\sup_{x\in \R^d}|L(x,[s,s+h])|^{\frac{p}{d(p-1)}}ds\right\}^{\frac{p-1}{p}}.
  \end{align}
  By the same calculations as above we get for all $t\in (0,1]$ and $-t<h<0$
    \begin{align}\label{X sup L1}
    &(1-|h|)|h|^{1/d} \nonumber\\
     &\leq \|s\mapsto\sup_{r,\tau\in[0,1]}\|f(hr+s)-f(h\tau+s)\|\|_{L^p(I(h),\R)}  \nonumber\\
     & \qquad\qquad\qquad\qquad\times \left\{\int_{-h}^{1}\sup_{x\in \R^d}|L(x,s)-L(x,s+h)|^{\frac{p}{d(p-1)}}ds\right\}^{\frac{p-1}{p}}.
  \end{align}
  According to \eqref{X sup L}, \eqref{X sup L1}, and \eqref{modulus continuity Besov Adler} we get for all $t\in (0,\frac{1}{2})$,
  \begin{align}\label{X sup L2}
    \frac{t^{1/d}}{2}
     &\leq c\, t^{\mu/d} \sup_{|h|\leq t}\|s\mapsto\sup_{r,\tau\in[0,1]}\|f(hr+s)-f(h\tau+s)\|\|_{L^p(I(h),\R)}.
  \end{align}
  Therefore,
  $$0<\lim_{t\to 0^+}t^{-(1-\mu)/d}\sup_{|h|\leq t}\|s\mapsto\sup_{r,\tau\in[0,1]}\|f(hr+s)-f(h\tau+s)\|\|_{L^p(I(h),\R)}.$$
  Then Theorem \ref{thm equivalance norm} (iv) and Theorem \ref{B0 equaal b0}  conclude the proof of Theorem \ref{Adler Besov}.
  \end{proof}
  \begin{proof}[Proof of Theorem \ref{thm non Besov X}]
 For any $1<p<\infty$, let $n=n(p)$ be a positive integer such that $d^n\geq \frac{p}{p-1}$. We have almost surely,
  \begin{align*}
     & \sup_{0<t\leq 1}t^{-(1-\alpha d)/d}\sup_{|h|\leq t}\|s\mapsto \sup_{x\in \R^d}|L(x,s+h)-L(x,s)|^{\frac{1}{d}}\|_{L^{\frac{p}{(p-1)}}(I(h);\R)} \\
     &\leq \sup_{0<t\leq 1}t^{-(1-\alpha d)/d}\sup_{|h|\leq t}\|s\mapsto \sup_{x\in \R^d}|L(x,s+h)-L(x,s)|\|^{\frac{1}{d}}_{L^{d^{n-1}}(I(h);\R)} \\
  \end{align*}
  Hence Theorem \ref{thm Besov regularity of L} and \ref{Adler Besov} finish the proof of Theorem \ref{thm non Besov X}.
\end{proof}
\begin{proof}[Proof of Corollary \ref{cor non Besov X p q}]
  It is a consequence of Theorem \ref{thm non Besov X} and the injections \eqref{injection Besov}.
\end{proof}
\section{Examples}\label{Examples}
\subsection{The Gaussian case}
Let $Y^0=(Y^0_t)_{t\in [0,1]}$ be a real-valued continuous centred Gaussian process, with $Y(0)=0$, that satisfies the classical local nondeterminism (LND) property on $[0,1]$. By \cite[Lemma 2.3]{Berman73} we have for any $m\geq2$, there exist two positive constants $c_m$ and $\varepsilon$ such that for every ordered points $0=t_0\leq t_1<\cdots<t_m\leq 1$  with $t_m-t_1<\varepsilon$, and $(v_1,\cdots,v_m)\in \R^m\setminus\{0\}$,
  \begin{equation}\label{LND}
    \var\left(\sum_{j=1}^{m}v_j(Y^0_{t_j}-Y^0_{t_{j-1}})\right)\geq c_m \sum_{j=1}^{m}v^2_j\var\left(Y^0_{t_j}-Y^0_{t_{j-1}}\right).
  \end{equation}
 Assume also that $Y^0$ verifies \eqref{minoration var}, with some $\alpha\in (0, 1)$. Define $Y_t=(Y^1_t,\cdots,Y^d_t),$ where $Y^1,\cdots,Y^d$ are independent copies of $Y^0$. Then $Y$ is $\alpha$-LND on $[0,1]$. The following theorem is a consequence of Theorem \ref{thm Besov regularity of L} and \ref{thm non Besov X}, and Corollary \ref{cor non Besov X p q}.
\begin{thm}\label{non Besov Gaussian}
  Let $Y^0=(Y^0_t)_{t\in [0,1]}$ be a real-valued continuous centered Gaussian process, with $Y(0)=0$, that satisfies the classical local nondeterminism (LND) property on $[0,1]$ and inequalities \eqref{minoration var} with $0<\alpha<\frac{1}{d} $. Let $Y^1,\cdots,Y^d$  be independent copies of $Y^0$ and put $Y_t=(Y^1_t,\cdots,Y^d_t)$.  Denote by $L(x,t)$ the jointly continuous version of the local time of $Y$, Therefore
    \begin{equation}\label{Gaussian Besov LT}
    \mathbb{P}\left[L(x,\bullet)\in \mathbf{B}^{1-d\alpha}_{p,\infty}(I;\R)\,,\;\text{ for all } x\in \R^d \text{ and } p\in [1,\infty)\right]=1;\qquad\;
  \end{equation}
  \begin{equation}\label{Gaussian non Besov process}
    \mathbb{P}\left[Y(\bullet)\in \mathbf{B}^{\alpha,0}_{p,\infty}(I,\R^d)\,,\;\text{ for some } p\in(1/\alpha,\infty)\right]=0;\qquad\qquad\qquad
  \end{equation}
   \begin{equation}\label{Gaussian non Besov process p q}
    \mathbb{P}\left[Y(\bullet)\in \mathbf{B}^{\alpha}_{p,q}(I,\R^d)\,,\;\text{ for some } p\in(1/\alpha,\infty) \text{ and } q\in [1,\infty)\right]=0,
  \end{equation}
  where $I=[0,1]$, $L(x,\bullet):t\in I\mapsto L(x,t)\in\R$, and $Y(\bullet):t\in I\mapsto Y_t\in \R^d$.
\end{thm}
A particular example is $ Y^0=B^{H,K} $ a bifractional Brownian motion with $H\in (0,1)$ and $K\in (0,1]$;
that is a real-valued centred Gaussian process, starting from zero, with covariance function
$$\E\left(B^{H,K}_t \, B^{H,K}_s\right)=\frac{1}{2^K}\left[(t^{2H} + s^{2H})^K -\vert t-s\vert^{2HK}
\right]\,.$$
Notice that the case $K=1$ corresponds to the fractional Brownian motion with Hurst parameter $H\in (0,1) $.
From  \cite[Lemma 3.3]{ait2019continuity} we know that the bifractional Brownian motion is LND and by \cite[Eq. (1)]{ait2019continuity} the Hypothesis \textbf{H} holds with $\alpha= HK$. Therefore, The below corollary is a consequence of Theorem \ref{non Besov Gaussian}.
\begin{cor}
  Let $(B^{H,K}_t)_{t\in [0,1]}$ be a $d$-dimensional bifractional Brownian motion with $H\in (0,1)$ and $K\in (0,1]$, s.t. $HK<\frac{1}{d}$. Denote by $L(x,t)$ the jointly continuous version of the local time of $B^{H,K}$, Therefore
    \begin{equation}\label{bifractional Besov LT}
    \mathbb{P}\left[L(x,\bullet)\in \mathbf{B}^{1-dHK}_{p,\infty}(I;\R)\,,\;\text{ for all } x\in \R^d \text{ and } p\in [1,\infty)\right]=1;
  \end{equation}
  \begin{equation}\label{bifractional non Besov process}
    \mathbb{P}\left[B^{H,K}(\bullet)\in \mathbf{B}^{HK,0}_{p,\infty}(I,\R^d)\,,\;\text{ for some } p\in(1/HK,\infty)\right]=0;\quad\,
  \end{equation}
   \begin{equation*}
    \mathbb{P}\left[B^{H,K}(\bullet)\in \mathbf{B}^{HK}_{p,q}(I,\R^d)\,,\;\text{ for some } p\in(1/HK,\infty) \text{ and } q\in [1,\infty)\right]=0,
  \end{equation*}
  where $I=[0,1]$, $L(x,\bullet):t\in I\mapsto L(x,t)\in\R$, and $B^{H,K}(\bullet):t\in I\mapsto B^{H,K}_t\in \R^d$.
\end{cor}
\begin{rem}
A different approach, based on characterization of Besov spaces in terms of sequences spaces, has been used in \cite{boufoussi2020besov} to investigate \eqref{bifractional non Besov process} for $d=1$.
However, to the best of our knowledge, the uniform Besov regularity results for the local times of Gaussian processes given by \eqref{Gaussian Besov LT} and \eqref{bifractional Besov LT} are new and have never been considered in the literature.
\end{rem}
\subsection{The non-Gaussian case}
Let us consider the following system of non-linear stochastic heat equations
\begin{equation}\label{SHE}
  \frac{\partial u_k}{\partial t}(t,x) = \frac{\partial^2 u_k}{\partial x^2}(t,x)+b_k(u(t,x))+\sum_{l=1}^{d}\sigma_{k,l}(u(t,x))\dot{W}^l(t,x),
\end{equation}
with Neumann boundary conditions
$$u_k(0,x)=0,\qquad \frac{\partial u_k(t,0)}{\partial x}=\frac{\partial u_k(t,1)}{\partial x}=0,$$
for $t\in [0,T]$, $1\leq k\leq d$, $x\in [0,1]$, where $u:=(u_1,\cdots,u_d)$. Let $\dot{W}=(\dot{W}^1,\cdots,\dot{W}^d)$ be a vector of
$d$-independent space-time white noises on $[0,T]\times [0,1]$.
Set $b=(b_k)_{1\leq k\leq d} $ and $\sigma=(\sigma_{k,l})_{1\leq k, l\leq d}$. We consider the below hypotheses on the coefficients $\sigma_{k,l}$ and $b_k$:\\
\textbf{A1}\; For all $1\leq k,l\leq d$,  $\sigma_{k,l}$ and $b_k$ are  bounded and infinitely differentiable functions with their partial derivatives of all orders are bounded.\\
\textbf{A2}\; The matrix $\sigma$ is uniformly elliptic, i.e. there exists $\rho>0$ such that for all $x\in \R^d$ and $y\in \R^d$ with $\|y\|=1$, we get $\|\sigma(x)y\|^2\geq \rho^2$ (where $\|\cdot\|$ is the Euclidean norm on $\R^d$).

Following Walsh \cite{walsh1986introduction}, a mild solution of \eqref{SHE} is a jointly measurable $\R^d$-valued process $u=(u_1,\cdots,u_d)$ such that for any $k\in \{1,\cdots,d\}$, $t\in [0,T]$, and $x\in [0,1]$,
\begin{equation}\label{solution SHE}
  \begin{split}
     u_k(t,x) &= \int_{0}^{t}\int_{0}^{1}G_{t-r}(x,v)\sum_{l=1}^{d}\sigma_{k,l}(u(r,v))W^l(dr,dv) \\
       &\qquad + \int_{0}^{t}\int_{0}^{1}G_{t-r}(x,v)b_k(u(r,v)) dv\, dr.
  \end{split}
\end{equation}

According to \cite[Theorem 5.4]{boufoussi2021local}, we know  that the solution to the system of non-linear stochastic heat equations, $\eqref{solution SHE}$, is $\frac{1}{4}$-LND, and by \cite[Eq. (2.12)]{boufoussi2021local} the hypothesis \textbf{H} holds with $\alpha=\frac{1}{4}$. Hence, Theorem \ref{thm Besov regularity of L} and \ref{thm non Besov X}, and Corollary \ref{cor non Besov X p q} give the following:
\begin{thm}
  Let $u$ be given by \eqref{solution SHE}.  Assume that $d\leq3$ and denote by $L(\xi,t)$ the jointly continuous version of the local time of the process $\{u(t,x)\,,\,\, t\in[0,T]\}$ for $x$ being fixed in $(0,1)$.Then
    \begin{equation}\label{nonlinear SHE Besov LT}
    \mathbb{P}\left[L(\xi,\bullet)\in \mathbf{B}^{1-d/4}_{p,\infty}(I;\R)\,,\;\text{ for all } \xi\in \R^d \text{ and } p\in [1,\infty)\right]=1;
  \end{equation}
  \begin{equation}\label{nonlinear SHE Besov process}
    \mathbb{P}\left[u(\bullet,x)\in \mathbf{B}^{1/4,0}_{p,\infty}(I,\R^d)\,,\;\text{ for some } p\in(4,\infty)\right]=0;\qquad\qquad\,
  \end{equation}
   \begin{equation}\label{nonlinear SHE Besov process p q}
    \mathbb{P}\left[u(\bullet,x)\in \mathbf{B}^{1/4}_{p,q}(I,\R^d)\,,\;\text{ for some } p\in(4,\infty) \text{ and } q\in [1,\infty)\right]=0,
  \end{equation}
  where $I=[0,1]$, $L(\xi,\bullet):t\in I\mapsto L(\xi,t)\in\R$, and $u(\bullet,x):t\in I\mapsto u(t,x)\in \R^d$.
\end{thm}
\begin{rem}
In \cite{boufoussi2020besov1}, we have studied, for $d=1$ and $\sigma=1$, i.e., $u(t,x)$ is the solution to the linear stochastic heat equation, by a different method the following:
\begin{equation*}
    \mathbb{P}\left[u(\bullet,x)\in \mathbf{B}^{1/4,0}_{p,\infty}(I,\R)\right]=0.
  \end{equation*}
 However, \eqref{nonlinear SHE Besov LT} is new even in the linear case.
\end{rem}

\bibliographystyle{emss}
\bibliography{biblio}

\end{document}